\documentclass[sigconf]{acmart}
\usepackage{booktabs} 

\usepackage[english]{babel}
\usepackage{amssymb,amsmath,amsthm,color}
\usepackage{hyperref}
\usepackage{enumitem}
\usepackage{algorithm}

\theoremstyle{definition}

\newtheorem{remark}[theorem]{Remark}

\DeclareMathOperator{\lc}{lc}
\DeclareMathOperator{\codim}{codim}
\DeclareMathOperator{\Rep}{Rep}

\newcommand{\x}{{\bf x}}
\newcommand{\y}{{\bf y}}

\newcommand{\z}{{\bf z}}
\newcommand{\C}{{\mathbb{C}}}

\newcommand{\Q}{{\mathbb{Q}}}

\newcommand{\I}{{\mathcal I}}

\begin{document}

\copyrightyear{2018}
\acmYear{2018}
\setcopyright{acmcopyright}
\acmConference[ISSAC '18]{2018 ACM International Symposium on Symbolic and Algebraic Computation}{July 16--19, 2018}{New York, NY, USA}
\acmPrice{15.00}
\acmDOI{10.1145/3208976.3208996}
\acmISBN{978-1-4503-5550-6/18/07}

\fancyhead{}

\title{Irredundant Triangular Decomposition}

\author{Gleb Pogudin}
\affiliation{%
  \institution{Courant Institute of Mathematical Sciences}
  \streetaddress{251 Mercer st.}
  \city{New York} 
  \state{NY} 
  \postcode{10012}
}
\email{pogudin@cims.nyu.edu}

\author{Agnes Szanto}
\affiliation{%
  \institution{North Carolina State University}
  \streetaddress{Campus Box 8205}
  \city{Raleigh} 
  \state{NC} 
  \postcode{27695}
}
\email{aszanto@ncsu.edu}

\begin{abstract}
Triangular decomposition is a classic, widely used and well-developed way to represent algebraic varieties with many applications.
In particular, there exist 
\begin{itemize}[leftmargin=\dimexpr\linewidth-8cm-\rightmargin\relax]
  \item sharp degree bounds for a single triangular set in terms of intrinsic data of the variety it represents,
  \item powerful randomized algorithms for computing triangular decompositions using Hensel lifting in the zero-dimensional case and for irreducible varieties. 
\end{itemize}
However, in the general case, most of the algorithms computing triangular decompositions  produce embedded components, which makes it impossible to
directly apply the intrinsic degree bounds. 
This, in turn, is an obstacle for efficiently applying Hensel lifting due to the  higher degrees of the output polynomials and the lower probability of success.

In this paper, we give an algorithm to compute an irredundant triangular decomposition of an arbitrary algebraic set $W$ 
defined by a set of polynomials in $\C[x_1, x_2, \ldots, x_n]$. 
Using this irredundant triangular decomposition, we  are able to give intrinsic degree bounds for the polynomials appearing in the triangular sets
and apply Hensel lifting techniques.
Our decomposition algorithm is randomized, and we analyze  the probability of success.

\end{abstract}



\maketitle
 
\section{Introduction}

Given a set of polynomials $f_1, \ldots, f_s \in \C[x_1, x_2, \ldots, x_n]$, consider the algebraic set 
\[
W = \{\z\in \C^n\;:\; f_1(\z)=\cdots=f_s(\z)=0\}.
\]
There are several common representations of algebraic sets that allow one to answer different questions about algebraic sets
or perform operations with them efficiently, for example representations via Gr\"obner bases, geometric resolution, and triangular decomposition.
This paper is focused on the latter.

Triangular decomposition is an important  tool with many applications,  
its origins going back to the works of Ritt~\cite{Ritt66}, who introduced the concept of characteristic sets. 
Several authors, including   Wu~\cite{Wu84}, Lazard~\cite{Lazard91}, Kalkbrener~\cite{Kalkbrener93}, Wang~\cite{Wang93}, 
Moreno Maza~\cite{Maza2000}, Schost~\cite{Schost2003b}, Chen~\cite{ChenThesis,ChenMaza12}, Dahan et al. \cite{Lifting}, 
have worked on triangular decompositions of algebraic sets, and some of these algorithms are implemented in {\sc Maple} in the package {\tt RegularChains}.
There exist sharp degree and height bounds for a single triangular set in terms of intrinsic data of the variety it represents, 
see for example the sequence of papers \setcitestyle{nosort}\cite{Schost2003,Schost2003b,DahanSchost,Dahanetal2012},
 these bounds are polynomial in the degree and the height of the variety. 
There are also  powerful randomized algorithms for computing triangular decompositions using Hensel lifting in the zero-dimensional case~\cite{Lifting}
and for irreducible varieties~\cite{Schost2003}. 

On the other hand, most of the algorithms computing triangular decompositions in the general case produce embedded components.
We are not aware of any easy way to delete all the embedded components afterward.
Moreover, the problem of checking inclusion between two algebraic sets defined by triangular sets is known as the algebraic version of the \emph{Ritt problem} (see \cite[p. 190]{Kolchin1973} and \cite[p. 44]{AlvandiThesis} for the algebraic version) and appears to be hard.
Embedded components make it impossible to directly apply the intrinsic degree bounds. 
The best known degree bounds for the polynomials in a triangular decomposition are essentially $D^{O(n^2)}$  \cite{GalloMishra1991,SzantoThesis,Amzallagetal2016} ($D$ is a bound on the total degrees of $f_1, \ldots, f_s$ ), 
which is not polynomial in the degree of the algebraic set represented by the triangular decomposition. 
As we show in the present paper, an irredundant triangular decomposition was needed to apply 
the intrinsic degree bounds of~\cite{DahanSchost} that are polynomial in the degree of the variety.
We note that one could achieve irredundant triangular decompositions by computing the irreducible components  of the variety and their Gr\"obner basis, which would allow one to factor out repeated and embedded components \cite{Wang92,Kalkbrener94}. 
However, this method is too expensive, for example, 
they require polynomial factorization and Gr\"obner basis computation 
with much higher worst-case degree bounds that we aim in this paper. 

We also mention that using random linear changes of the variables one can avoid embedded components and compute an irredundant equidimensional decomposition, as demonstrated in \cite{JeronimoSabia,Lecerf2003}. However, changing the coordinate system destroys the triangular structure in the original variables, and, in particular, does not allow to perform elimination of some of the original variables.  We  use \cite{JeronimoSabia} in the present paper as one of the subroutines, but in a way that our final output does not use coordinate transformations. 

As far as we know, irredundant decomposition using triangular sets, without random changes of variables,  
was not known previously. There are two difficulties:

\begin{enumerate}[leftmargin=\dimexpr\linewidth-8.2cm-\rightmargin\relax]
\vspace{-.1cm}
  \item  The first difficulty is to detect common irreducible components among triangular sets of the same dimension 
  but with different sets of {\em free variables} (see Definition~\ref{def:triang}),
  because even in the equidimensional case we may need to compute triangular sets using different sets of free variables. 
  As far as we know, there were no previous methods to detect if two such representations have a common irreducible component or not. 
  One of the main results of this paper is a new technique that ensures that triangular sets representing equidimensional 
  components with different sets of free variables have no common irreducible components 
  (see Step~\ref{step:avoid} of Algorithm~\ref{alg:main} and Lemma~\ref{lemma:irredundant}). 

  \item The second difficulty is to  factor out components that are embedded in higher dimensional irreducible components, 
  similarly as it is stated in the Ritt problem, mentioned above. 
  This problem has only been solved for triangular sets in one and two dimensions \cite{Chenetal2013,AlvandiThesis}. 
  The second result of this paper is that we show how to use the results in \cite{JeronimoSabia} and turn their 
  irredundant equidimensional decomposition into an irredundant triangular decomposition. 
  To do that, we use the zero-dimensional {\em equiprojectable} decomposition of \cite{Lifting} and the lifting techniques of \cite{Schost2003}. 
\end{enumerate}

\section{Main result}

For $T\subset\C[\x]$, $Z(T)\subset \C^n$ denotes the set of common roots of $T$. For $V\subset \C^n$, $I(V)$ denotes the set of polynomials in $\C[\x]$ vanishing on $V$.
We recall some definitions from~\cite{Hubert1}.

\begin{definition}\label{def:triang}
  A set of polynomials $\Delta$ of the form 
  \begin{equation}\label{eq:Delta}
  \Delta = \{g_1({\bf y}, z_1), g_2({\bf y}, z_1, z_2), \ldots,g_m({\bf y}, z_1, \ldots, z_m)\}\subset \C[{\bf x}],
  \end{equation}
  where ${\bf  y}= y_1, \ldots, y_d$,  $\{x_1, \ldots, x_n\}=\{{\bf y}, z_1, \ldots, z_m\}$,  $\;d+m=n$, and $g_i$ involves $z_i$ for every $1 \leqslant i \leqslant m$
  is said to be a \emph{triangular set}. 
  
  The variables $\mathbf{y}$ are called \emph{free variables}.
  For every $1 \leqslant i \leqslant m$, $z_i$ is said to be \emph{the leader} of $g_i$, and
  we denote the leading coefficient of $g_i$, viewed as a univariate polynomial in $z_i$, by ${\rm lc}(g_i)$.
  Let $I_{\Delta}:=\{{\rm lc}(g_k)\;:\; k=1, \ldots, m\}$.
  The ideal generated by  $\Delta$ in $\C[\x]$ is denoted by $(\Delta)$, and the ideal ``pseudo-generated" by $\Delta$ is the saturated ideal
  \[
    \I(\Delta):=(\Delta):I_\Delta^\infty.
  \]
  $\Rep(\Delta)$ denotes the affine variety represented by a triangular set $\Delta$, defined as 
  \[
  \Rep(\Delta):=Z(\I(\Delta))=\overline{Z(\Delta)\setminus Z(\prod_k {\rm lc}(g_k))}\subset \C^{n}. 
  \]
\end{definition}

\begin{definition}
  The triangular set~\eqref{eq:Delta} is called a {\em regular chain} if
  \begin{itemize}
    \item $\deg_{z_i}(g_j)<\deg_{z_i}(g_i)$ for every $1 \leqslant i < j \leqslant m$;
    \item $\lc(g_k)$ is not a zero divisor in $\C[\y,z_1, \ldots, z_{k-1}]/\I(g_1, \ldots, g_{k-1})$ for every $1 \leqslant k \leqslant m$.
  \end{itemize}
  
  The regular chain $\Delta$ is called a {\em square-free regular chain} if it is a regular chain and 
  $g_k$ is a square-free polynomial in $z_k$ over $\C[\y,z_1, \ldots, z_{k-1}]/\I(g_1, \ldots, g_{k-1})$
  for every $1 \leqslant k \leqslant m$.
\end{definition}

The main result of the paper is the following.

\begin{theorem}\label{theorem:main}
  Let $W = Z(f_1, \ldots, f_s)$ for $f_i \in \mathbb{C}[x_1, x_2, \ldots, x_n]$ for $i = 1, \ldots s$. 
  Assume that the total degree of $f_i$ does not exceed $D \geqslant 2$ for every $1 \leqslant i \leqslant s$.
  We give a randomized algorithm (Algorithm~\ref{alg:main}) that computes  
  \[
  \mathbf{T} = \left\{\Delta_{i}\;:\; 1 \leqslant i \leqslant N \right\}
  \]
  such that 
  \begin{enumerate}
    \item \label{theorem(1)}$\Delta_{i}$ is a {\em square-free regular chain} for every $1 \leqslant i \leqslant N$;
    \item \label{theorem(2)} None of the irreducible components of $\;\Rep(\Delta_{i})$ is contained in $\;\Rep(\Delta_{j})$ for $i \neq j$;
    \item \label{theorem(3)}$W = \bigcup\limits_{i = 1}^N \Rep(\Delta_i)$;
    \item \label{theorem(4)} For every $1 \leqslant i \leqslant N$ and every $g \in \Delta_i$, the total degree of $g$ with respect to the free variables does not exceed $(\deg W)^2$ and the degree with respect to every other variable does not exceed $\deg W$;
    \item\label{theorem(5)} All polynomials appearing in Algorithm~\ref{alg:main} have total degrees  bounded by 
    \[
    \max\left( (n+1)D^{n+1}, D^{2n} + D^n\right);
    \]
    \item\label{theorem(6)} Assuming that in Algorithm~\ref{alg:main} we make random choices independently and uniformly
    from a finite subset $\Gamma\subset\C$,  
    the probability that the output of Algorithm~\ref{alg:main} is correct is at least 
    \[
      1 - \frac{cD^{n^2 + n} + (n + 1)^4D^{c'(n + 1)}}{|\Gamma|}
    \]
    for some constants $c$ and $c'$.
  \end{enumerate}
  
\end{theorem}

\begin{remark}
 One can check that our algorithm uses only gcd computation and linear algebra, so if the input polynomials are over a subfield $k \subset \C$ (for example, $k = \mathbb{Q}$), then the output polynomials will be also over this subfield. 
 \end{remark}


\section{The Toolbox}

\subsection{Notation} \label{notation:1}

Let $\x=(x_1, \ldots,x_n)$, and for a subset $S  = \{i_1, \ldots, i_m \} \subset \{ 1, \ldots, n\}$, denote  $\mathbf{x}_{S} := ( x_{i_1}, \ldots, x_{i_m}) $ with $i_1<i_2<\cdots<i_m$.


\subsection{Equidimensional decomposition}\label{subsec:JeronimoSabia}
We use the method  in~\cite{JeronimoSabia} to eliminate embedded components in our  Main Algorithm.  The main idea of~\cite{JeronimoSabia} to avoid embedded components is to represent each equidimensional part as the intersections of   $n+1$ hypersurfaces, each a Chow form with respect to a random coordinate system. Then they use the equations of the higher dimensional parts to factor out lower dimensional embedded components. Algorithm \ref{alg:equidim} below is the input and output specification of the algorithm in~\cite{JeronimoSabia}. 

{\small\begin{algorithm}[h]
\caption{$\operatorname{EquiDim}(f_1, \ldots, f_s)$}\label{alg:equidim}
\begin{description}[leftmargin=\dimexpr\linewidth-8cm-\rightmargin\relax]
  \item[Input]  $f_1, \ldots, f_s \in \mathbb{C}[x_1, \ldots, x_n]$, defining an algebraic variety $W = Z(f_1, \ldots, f_s)$ and a real number $0 < p < 1$.

 \item[Output]  The sets  $$\mathbf{p}_0 = \{p_{0, 0}, \ldots, p_{0, n}\}, \ldots,\mathbf{p}_{n - 1} = \{p_{n - 1, 0}, \ldots, p_{n - 1, n}\}$$
of polynomials in $\mathbb{C}[x_1, \ldots, x_n]$ represented by straight-line programs such that with probability at least $p$ the following holds
\begin{itemize}[leftmargin=\dimexpr\linewidth-8cm-\rightmargin\relax]
  \item the set $Z(p_{\ell, 0}, \ldots , p_{\ell, n})\subset \C^n$ is exactly $W_\ell$, that is the union of all irreducible components of $W$ of dimension $\ell$, 
  for all $0 \leqslant \ell \leqslant n - 1$;
  \item $\deg p_{\ell, i} \leqslant \deg W_\ell$ for all $0 \leqslant \ell \leqslant n - 1$ and $0 \leqslant i \leqslant n$.
\end{itemize}
\end{description}
\end{algorithm}}


\subsection{Canny's generalized resultant}\label{subsec:Canny}
Consider polynomials $f_1, \ldots, f_{n + 1} \in \mathbb{C}[\mathbf{x}, \mathbf{y}]$, 
where $\mathbf{x} = (x_1, \ldots, x_n)$ and $\mathbf{y} = (y_1, \ldots, y_m)$.
Let $\deg f_i \leqslant D$ for every $1 \leqslant i \leqslant n+1$.
By $\pi\colon \mathbb{C}^{m + n} \to \mathbb{C}^m$ we denote the projection of the $(\mathbf{x}, \mathbf{y})$-space onto the $\mathbf{y}$-coordinates.
Then the construction of the generalized perturbed resultant proposed in~\citep{Canny} provides 
a non zero polynomial $\operatorname{PRes}_{\mathbf{x}}(f_1, \ldots, f_{n + 1}) \in \mathbb{C}[\mathbf{y}]$ such that
\begin{itemize}[leftmargin=\dimexpr\linewidth-8.3cm-\rightmargin\relax]
  \item\label{Canny:vanishing} $\operatorname{PRes}_{\mathbf{x}}(f_1, \ldots, f_{n + 1})$ vanishes on $\pi(C)$ for every irreducible  component $C \subset Z(f_1, \ldots, f_{n+1})$ with $\overline{\pi(C)} \neq \mathbb{C}^m$;
  \item $\deg \operatorname{PRes}_{\mathbf{x}}(f_1, \ldots, f_{n + 1}) \leqslant (n + 1)D^{n + 1}$ 
  (this follows from the degree bound for multivariate resultants~\citep[Proposition~1.1]{GKZ});
  \item $\operatorname{PRes}_{\mathbf{x}}(f_1, \ldots, f_{n + 1})$ can be computed using one multivariate resultant computation for $n + 1$ polynomials of degree at most $D$.  
\end{itemize}
         
\subsection{Triangular Decomposition over Fraction Fields}\label{subsec:triangularzerodim}

A randomized algorithm $\operatorname{TriangularZeroDim}$ with the following specifications 
will be used as a subroutine in Algorithm~\ref{alg:main}.
One possible way to design such an algorithm  is based on the equiprojectable triangular decomposition algorithm of \cite{Lifting} and is described  in Section~\ref{sec:lifting}.
There are also other possibilities such as the $\operatorname{unmixed}$ procedure from~\cite{SzantoThesis}.\\

{\small\begin{algorithm}[h]
\caption{$\operatorname{TriangularZeroDim}(S, \{h_1, \ldots, h_\ell\})$}\label{alg:triangularzerodim}
\begin{description}[leftmargin=\dimexpr\linewidth-8cm-\rightmargin\relax]
  \item[Input] 
  \begin{itemize}[leftmargin=\dimexpr\linewidth-7.8cm-\rightmargin\relax]
  \item a proper subset $S := \{i_1, \ldots, i_m\} \subset \{1, \ldots, n\}$ with $i_1 < \ldots < i_m$ 
    \item polynomials $h_1, \ldots, h_\ell \in \mathbb{C}[\mathbf{x}]$ with $\ell \geqslant m$ such that
    \begin{enumerate}[leftmargin=\dimexpr\linewidth-7.5cm-\rightmargin\relax]
      \item the ideal $I$ generated by $h_1, \ldots, h_\ell$ in $\mathbb{C}(\mathbf{x}_{\overline{S}})[\mathbf{x}_{S}]$ is zero dimensional
      \item the Jacobian of $h_1, \ldots, h_m$ with respect to $\mathbf{x}_S$ is invertible at every solution of $I$ in the algebraic closure $\overline{\mathbb{C}(\mathbf{x}_{\overline{S}})}$
    \end{enumerate}
  \end{itemize}
  
  \item[Output] Square-free regular chains $\Delta_1, \ldots, \Delta_q \subset \mathbb{C}[\mathbf{x}]$ such that
  \begin{enumerate}[leftmargin=\dimexpr\linewidth-8cm-\rightmargin\relax]
    \item leaders of $\Delta_j$ are $x_{i_1}, \ldots, x_{i_m}$ for every $1 \leqslant j \leqslant q$
    \item $\Rep(\Delta_i)$ and $\Rep(\Delta_j)$ do not have common irreducible components for every $1 \leqslant i < j \leqslant q$
    \item $\bigcap\limits_{i = 1}^q \I(\Delta_i)\cdot \mathbb{C}(\mathbf{x}_{\overline{S}}) = \sqrt{I}$
    \item for every $g \in \Delta_1 \cup \ldots \cup \Delta_q$, the coefficients of $g$ considered as a polynomial in $\mathbf{x}_S$ are coprime. 
  \end{enumerate}
\end{description}
\end{algorithm}}

Note that the algorithm in Section~\ref{sec:lifting} returns triangular sets such that the leading coefficients of their elements belong to $\C[\mathbf{x}_{\overline{S}}]$. This property is not needed in the proof of correctness of Algorithm~\ref{alg:main}, but we use it to prove our degree bounds. 


\section{The main algorithm} 

See our Main Algorithm, Algorithm~\ref{alg:main}, below. 

{\small \begin{algorithm}[h]
\caption{Main Algorithm}\label{alg:main}
\begin{description}
  \item[Input] $f_1, \ldots, f_s \in \mathbb{C}[x_1, x_2, \ldots, x_n]$ defining the affine variety $W=Z(f_1, \ldots, f_s)$.
  \item[Output] Representation of $W$ as a union of varieties defined  by square-free regular chains, as described in Theorem~\ref{theorem:main}.
\end{description}

\vspace{3mm}

\begin{enumerate}[leftmargin=\dimexpr\linewidth-8cm-\rightmargin\relax]

  \item \label{step:equidim}\emph{Compute the equidimensional decomposition.} 
  For every $0 \leqslant d < n$, we call the subroutine $\operatorname{EquiDim}(f_1, \ldots, f_s)$ described in Section~\ref{subsec:JeronimoSabia} to compute a set of polynomials 
  \[
  \mathbf{p}_d := \{p_{d, 0}, \ldots, p_{d, n }\}
  \]
  such that 
  \[
  Z(\mathbf{p}_d) = W_d,
  \]
  where $W_d$ is the union of  irreducible components of dimension $d$ in $W$. 
  Let $d_0$ and $d_1$ be the minimal and maximal dimensions, respectively.
  
  \item \emph{Compute a cover by ``univariate'' triangular sets.}\label{step:cover}
  \begin{enumerate}[leftmargin=\dimexpr\linewidth-7.5cm-\rightmargin\relax]
 
    \item\label{step:square} \emph{Square the system.} Let 
    \[
    \widetilde{f}_i := \lambda_{i, 1}f_1 + \ldots + \lambda_{i, s} f_s \text{ for } 1 \leqslant i \leqslant n-d_0+1,
    \]
    where $\lambda_{i, j}$ is chosen uniformly random  from a finite set $\Gamma\subset \C$.
    \item \emph{Compute projections.}\label{step:Canny} For each $S \subset \{1, \ldots, n\}$ such that $n-d_1-1\leq |S|\leq n-d_0-1$, compute Canny's generalized resultants
    \[
    \hat{g}_S := \operatorname{PRes}_{\mathbf{x}_S}(\widetilde{f}_1, \ldots, \widetilde{f}_{m })\in \C[\mathbf{x}_{\overline{S}}] \subset \C[\mathbf{x}],
    \]
    defined  in Section \ref{subsec:Canny}, where $|S|=m-1$.
    \item \emph{Define the cover.} For each $S \subset \{1, \ldots, n\}$ such that $n-d_1\leq |S|\leq n-d_0$ define
    \[
    \widehat{\nabla}_S := \{ \hat{g}_{S_1}, \ldots, \hat{g}_{S_m} \},
    \]
    where $S = \{i_1, \ldots, i_m\} \subset \{1, \ldots, n\}$ with $i_1 < \ldots < i_m$, and $S_j := S \setminus \{ i_j \}$
    \item \label{step:avoid} \emph{Avoid repetitions.} 
    Choose a random point $\boldsymbol{\alpha} := (\alpha_1, \ldots, \alpha_n) \in \Gamma^n$ such that  $\hat{g}_S(\boldsymbol{\alpha})\neq 0$ for all $S\subset\{1, \ldots, n\}$.
    
    For every $n - d_1 \leqslant m \leqslant n - d_0$, for every subset $S = \{i_1, \ldots, i_m\} \subset \{1, \ldots, n\}$ with $i_1 < \ldots < i_m$, let
    \[
      \nabla_S := \{g_{S, 1}, \ldots, g_{S, m}\},
    \]
    where $g_{S, j}$ is the squarefree part of  $$\hat{g}_{S_j} \left(\alpha_1, \ldots, \alpha_{i_j - 1}, x_{i_j}, \ldots, x_n\right)$$ for $1 \leqslant j \leqslant m$, considered as a univariate polynomial in $x_{i_j}$. 
  \end{enumerate}
  
  \item \label{step:final} \emph{Compute the result.} Return
  \[
\bigcup_{d=d_0}^{d_1} \; \bigcup\limits_{\substack{S \subset \{1, \ldots, n\},\\ |S|=n-d}} \operatorname{TriangularZeroDim}(S, \nabla_S \cup \mathbf{p}_{d})
  \]
  The subroutine $\operatorname{TriangularZeroDim}$ is described in Section~\ref{subsec:triangularzerodim}.
\end{enumerate}
\end{algorithm}
}


\begin{proof}[{\bf  Proof of Theorem~\ref{theorem:main} \eqref{theorem(1)}-\eqref{theorem(5)}.}]
Denote by $\{ \Delta_1, \ldots, \Delta_N \}$ the output of Algorithm~\ref{alg:main}. 
We make the following assumptions on random choices made in Algorithm~\ref{alg:main} (the probabilities will be estimated in the proof of~\eqref{theorem(6)} in Section \ref{sec:lifting}):
\begin{enumerate}[label=A\arabic*]
  \item\label{A1} The choice of $\lambda_{i, j}$ in Step~\ref{step:square} satisfies the following property: for every $-1 \leqslant d \leqslant n - 1$,  $Z(\widetilde{f}_1, \ldots, \widetilde{f}_{n - d})$ and $W$ have the same irreducible components of dimensions larger than $d$.
  \item\label{A2} The point $\boldsymbol{\alpha}$ in Step~\ref{step:avoid} is chosen such that $\hat{g}_S(\boldsymbol{\alpha}) \neq 0$  for all $S\subset\{1, \ldots, n\}$.
\end{enumerate}

\underline{Proof of~\eqref{theorem(1)}.} The fact that $\Delta_i$ is a squarefree regular chain 
would follow from the specification of $\operatorname{TriangularZeroDim}$ if we show that the input specification of $\operatorname{TriangularZeroDim}$ is satisfied. We fix $S \subset \{1, \ldots, n\}$ of cardinality $n - d$.
Then the first $n - d$ polynomials in the input of
\[
\operatorname{TriangularZeroDim}(S, \nabla_S \cup \mathbf{p}_d)
\]
in Step~\ref{step:final} are $\nabla_S$.
These polynomials already generate a zero-dimensional ideal in $\C(\mathbf{x}_{\overline{S}})[\mathbf{x}_S]$, so $\nabla_S \cup \mathbf{p}_d$ also do.
Since $g_{S, j}$ belongs to $\C[x_{i_{j}}, \mathbf{x}_{\overline{S}}]$, the Jacobian of $\nabla_S$ with respect to $\mathbf{x}_{S}$ is a diagonal matrix.
Moreover, since every $g_{S, j}$ is squarefree, the matrix is invertible at every solution of $(\nabla_S)$ in $\overline{\C(\mathbf{x}_{\overline{S}})}$.

\underline{Proof of~\eqref{theorem(2)}.} Let $C$ be an irreducible component of $\Rep(\Delta_i)$. Below, 
Lemma~\ref{lemma:inW}  implies that $C$ is an irreducible component of $W$.
Using~\ref{A1} and Lemma~\ref{lemma:irredundant}, we conclude that $C$ is contained in $\Rep(\Delta_j)$ only for $j = i$. 
This proves the statement.

\underline{Proof of~\eqref{theorem(3)}.}
Lemma~\ref{lemma:inW} implies that $\bigcup\limits_{i = 1}^N\Rep(\Delta_i) \subset W$.
Lemma~\ref{lemma:irredundant} together with~\ref{A1} imply that $W \subset \bigcup\limits_{i = 1}^N\Rep(\Delta_i)$.

\underline{Proof of~\eqref{theorem(4)}.}
Fix $S = \{i_1, \ldots, i_m\}$ with $i_1 < \cdots < i_m$, and consider the coordinate system 
$\x_S=(x_{i_1}, \ldots, x_{i_{m}})=:(z_1, \ldots, z_m)$  and $\y:=\x_{\overline{S}}$. 
First we prove that if $V$ is represented by a square-free regular chain in the fixed coordinate system, then this square-free regular chain   
representing $V$ is unique, as long as the leading coefficients are in $\C[\y]$ and the coefficients in $\C[\y]$ of each polynomials are relatively prime. This is because for any such square-free regular chain, after dividing by the leading coefficients, we get a reduced Gr\"obner basis with respect to the lexicographic monomial ordering with $z_1<\cdots<z_m$ of the  ideal generated by $I(V)$ in the ring $\C(\y)[\z]$. Since the reduced  Gr\"obner basis of an ideal with a fixed monomial ordering is unique, using the assumption that the coefficients of each polynomial in the regular chain   are relatively prime, we get that the square-free regular chain representing $V$ that satisfy the above conditions is unique. 

Statement~\eqref{theorem(4)} of Theorem~\ref{theorem:main} follows from the fact that, as described in Section \ref{sec:lifting}, $\operatorname{TriangularZeroDim}(S,\nabla_S \cup \mathbf{p}_{d})$ 
returns a set $\{\Delta_a, \ldots, \Delta_b\}$ of square-free regular chains  such that the leading terms are in $\C[x_{\overline{S}}]$ and 
the coefficients of the polynomials in the triangular sets are relatively prime. 
Moreover, the output of $\{\Delta_a, \ldots, \Delta_b\}$ is an irredundant triangular decomposition of $W_S$, where $W_S$ is the union of all irreducible components $C$ of $W$ of 
co-dimension $m$ such that $\x_{\overline{S}}$ is the maximal subset of $\{x_1, \ldots, x_n\}$ among the subsets of free variables for $C$, 
with respect to the lexicographic ordering of the variables  $x_1,\ldots, x_n$.  
So for each $i=a, \ldots, b$,  $V := \Rep(\Delta_i)$ is a disjoint union of irreducible components of $W_S$. 
Finally we note that $\Delta_i=\{g_1, \ldots, g_m\}$ is the unique square-free regular chain representing $V$ with leading coefficients in $\C[\y]$, 
so we can apply the degree bounds proved in \cite[Theorem 2]{DahanSchost} to get for $k=1, \ldots, m$ 
\begin{eqnarray}\label{freedegree}
\deg_\y(g_k)\leq \left(1+2\sum_{i\leq k-1}(d_i-1)\right)\deg(V_k)\leq \deg(W_{S,k})^2\
\end{eqnarray}
where $d_i:=\deg_{z_i}(g_i)$, and $V_k$ ($W_{S,k}$) is the projection of $V$ ($W_S$) to the coordinates $(\y, z_1, \ldots, z_k)$.
Since $\deg(W_{S,k})\leqslant \deg(W)$, we get the desired bound for the free variable.

For the non-free variables, we use inequalities 
\[
\prod\limits_{i = 1}^m \deg_{z_i} g_i \leqslant \deg W \text{ and } \deg_\mathbf{z} g_k \leqslant \sum\limits_{i = 1}^m \deg_{z_i} g_i - m + 1
\]
to deduce $\deg_{\mathbf{z}} g_k \leqslant \deg W$ for $k = 1, \ldots, m$.

\underline{Proof of~\eqref{theorem(5)}.}
Due to Section~\ref{subsec:JeronimoSabia}, the degrees of the polynomials appearing in Step~\ref{step:equidim} of the algorithms do not exceed $\deg W \leqslant D^n$.

The bounds on the degrees of the polynomials in Step~\ref{step:cover} of the algorithm are bounded by $(n+1)D^{n+1}$, the bound on the degree of Canny's generalized resultant (see Section~\ref{subsec:Canny}).

As we show in Section~\ref{sec:lifting}, we can use Hensel lifting in the $\x_{\overline{S}}$ variables to compute the output of 
$\operatorname{TriangularZeroDim}(S,\nabla_S \cup \mathbf{p}_{d})$, thus the $\x_{\overline{S}}$-degrees and $\x_{{S}}$-degrees of the polynomials computed in this subroutine  
do not exceed the bounds $D^{2n}$ and $D^n$, respectively, stated in~\eqref{theorem(4)}.
\end{proof}

 \begin{lemma}\label{lemma:inW} Let $C$ be an irreducible component of $\;\Rep(\Delta_i)$ for some $1 \leqslant i \leqslant N$. 
  Then $C$ is an irreducible component of $W$.
  \end{lemma}
  
  \begin{proof}
  Since $\Delta_i$ belongs to the output of Algorithm~\ref{alg:main}, $\Delta_i$ belongs to 
  \[
  \operatorname{TriangularZeroDim}(S, \{g_{S, 1}, \ldots, g_{S, m}\}\cup \mathbf{p}_{n - m})
  \]
  for some $S := \{i_1, \ldots, i_m\} \subset \{1, \ldots, n\}$.
  The specification of $\operatorname{TriangularZeroDim}$ implies that 
  \begin{enumerate}[leftmargin=\dimexpr\linewidth-8cm-\rightmargin\relax]
    \item\label{stat:dimension} $|\Delta_i| = |S|=m$, so $\dim C = n - m$ by~\cite[Theorem~4.4]{Hubert1};
    \item\label{stat:intersection} leaders of $\Delta_i$ are $x_{i_1}, \ldots, x_{i_m}$, so $I(C) \cap \mathbb{C}[\mathbf{x}_{\overline{S}}] = 0$ by~\cite[Proposition~5.8]{Hubert1};
    \item\label{stat:membership} the ideal generated by $\I(\Delta_i)$ in $\mathbb{C}(\mathbf{x}_{\overline{S}})[\mathbf{x}_S]$ contains $\mathbf{p}_{n - m}$.
  \end{enumerate}
  Due to~\eqref{stat:membership}, there exists $q_j \in \mathbb{C}[x_{\overline{S}}]$ such that $q_j p_{n - m, j} \in \I(\Delta_i) \subset I(C)$ for every $0 \leqslant j \leqslant n $.
  Since $I(C)$ is prime and $I(C) \cap \mathbb{C}[x_{\overline{S}}] = 0$ due to~\eqref{stat:intersection}, $p_{n - m, j} \in I(C)$ for every $0 \leqslant j \leqslant n $.
  Since $W_{n-m}=Z(\mathbf{p}_{n - m})$, we have  $C \subset W_{n - m}$.
  Moreover, $C$ is an irreducible component of $W_{n - m}$ since $\dim C = n - m$ due to~\eqref{stat:dimension}.
  This proves the lemma.
 \end{proof}

\begin{lemma} \label{lemma:irredundant}
Assume that Assumptions~\ref{A1} and \ref{A2} are satisfied.
Then for every irreducible component $C \subset W$, there exists a unique $1 \leqslant i \leqslant N$ such that $C \subset \Rep(\Delta_i)$. Moreover, $C$ is an irreducible component of  $\Rep(\Delta_i)$.
\end{lemma}

\begin{proof}
  \textbf{Existence.}
  Let $C $ be an irreducible component of $W$,  $m := \codim C$ and $d := n - m$.
  Consider all subsets $\{ j_1, \ldots, j_d \} \subset \{1, \ldots, n\}$ such that the image of 
  $x_{j_1}, \ldots, x_{j_d}$ constitute a transcendence basis of $\mathbb{C}[\mathbf{x}]$ modulo $I(C)$.
  Among all these sets we find one for which the tuple $(j_1, \ldots, j_d)$ 
  (assuming that $j_1 < j_2 < \ldots < j_d$) is maximal with respect to the lexicographic ordering.
  By $S := \{i_1, \ldots, i_m\}$ we denote the complement to this subset in $\{1, \ldots, n\}$.
  
  Since $\dim C = d$ and $W \subset Z(\widetilde{f}_1, \ldots, \widetilde{f}_{n-d})$, our assumption that $Z(\widetilde{f}_1, \ldots, \widetilde{f}_{n-d})$ and $W$ have the same irreducible components of dimension larger than $d$ implies that $C$ is a (non-embedded) irreducible component of $Z(\widetilde{f}_1, \ldots, \widetilde{f}_{n-d})$. 
  Since $\dim C = d$, $C$ does not project dominantly 
  on the $\mathbf{x}_{\overline{S} \cup \{i_k\}}$-coordinates for  $1 \leqslant k \leqslant m$.
  Then the property~\ref{Canny:vanishing} of the Canny's resultant implies that 
  $\hat{g}_{S_k}$ (defined in Step \ref{step:Canny} with $S_k := S \setminus \{i_k\}$) 
  vanishes on $C$ for every $1 \leqslant k \leqslant m$.
  Consider $1 \leqslant k \leqslant m$. We will prove that $g_{S, k}$ (defined in  Step \ref{step:avoid}) vanishes on $C$.
  Let 
  \[
  B := \{ \ell \;|\; x_{\ell} \text{ appears in } \hat{g}_{S_k},\; \ell < i_k\}.
  \]
  If $B = \emptyset$, then $g_{S, k} = \hat{g}_{S_k}$ and vanishes on $C$.
  Otherwise, we write $\hat{g}_{S_k}$ as $\sum\limits_{i = 1}^M c_im_i$, where
  $m_0, \ldots, m_M$ are distinct monomials in $\mathbf{x}_B$ with $m_0 = 1$
  and $c_0, \ldots, c_M$ are polynomials in $\mathbb{C}[\mathbf{x}_{\overline{B}}]$.
  If all $c_0, \ldots, c_M$ vanish on $C$, then $g_{S, k}$ vanishes on $C$.
  Otherwise, there exists $1 \leqslant t \leqslant M$ such that $c_t$ does not vanish on $C$.
  Let $x_{j_\ell}$ be any variable in $m_t$, then $\hat{g}_{S_k} = 0$ is a nontrivial algebraic equation
  for $x_{j_\ell}$ over $\x_{\overline{S}} \setminus \{x_{j_\ell}\} \cup \{x_{i_k}\}$ modulo $I(C)$.
  Then, replacing $x_{j_\ell}$ with $x_{i_k}$, we obtain a lexicographically larger transcendence basis of $\mathbb{C}[\mathbf{x}]$
  modulo $I(C)$.
  Thus $g_{S, k}$ vanishes on $C$.
  
 Let 
 $
 \{\Delta_a, \ldots, \Delta_b\}= \operatorname{TriangularZeroDim}(S, \nabla_S\cup \mathbf{p}_d) 
$ for some  $1 \leqslant a \leqslant b \leqslant N$
 with $\nabla_S=\{g_{S, 1}, \ldots, g_{S, m}\}$. The specification of $\operatorname{TriangularZeroDim}$
  together with the fact that $\{ g_{S, 1}, \ldots, g_{S, m}\}\cup \mathbf{p}_d  \subset I(C)$ imply
  \[
  \bigcap\limits_{i = a}^b  \mathbb{C}(\mathbf{x}_{\overline{S}}) \I(\Delta_i)  = \sqrt{I} \subset  \mathbb{C}(\mathbf{x}_{\overline{S}})  I(C) .
  \]
  Since the latter ideal is prime, there exists $a \leqslant c \leqslant b$ such that 
  \[
  \mathbb{C}(\mathbf{x}_{\overline{S}}) \I(\Delta_c) \subset \mathbb{C}(\mathbf{x}_{\overline{S}}) I(C).
  \]
  Since $I(C) \cap \mathbb{C}[\mathbf{x}_{\overline{S}}] = 0$, we have $\I(\Delta_c) \subset I(C)$.
  Hence $C \subset \Rep(\Delta_c)$.
  Since $\dim C = \dim \Rep(\Delta_c)$, $C$ is an irreducible component of $\Rep(\Delta_c)$.
  The existence is proved.
  
 \noindent \textbf{Uniqueness.}
  Assume that $C \subset \Rep(\Delta_{c'})$ for some $c' \neq c$. Consider the following cases
  
 {\bf Case~1.} $\dim \Rep(\Delta_{c'})=D > d = \dim C$ Since $\Rep(\Delta_{c'}) \subset W_D$
    and $C$ is an irreducible  component of $W_d$ by Lemma~\ref{lemma:inW}, $C$ cannot be contained in $\Rep(\Delta_{c'})$.
    
{\bf Case~2.} $\dim \Rep(\Delta_{c'})=d = \dim C$
    Let
    \[
    \Delta_{c'} \subset \operatorname{TriangularZeroDim}(S', \{g_{S', 1}, \ldots, g_{S', m}\}\cup \mathbf{p}_d)
    \]
    for some $S' = \{i_1', \ldots, i'_m\} \subset \{1, \ldots, n\}$.
    Consider an arbitrary $1 \leqslant k \leqslant n$, and define $\overline{S}_k:=\overline{S} \cap \{k, \ldots, n\}$.
    Since $I(C)$ contains $g_{S, 1}, \ldots, g_{S, m}$, 
    $\x_{\overline{S}_k}$ is a transcendence basis
    of $\mathbb{C}[x_k, \ldots, x_n]$ modulo $I(C)\cap \C[x_k, \ldots, x_n]$, since $\x_{\overline{S}_k}$ is clearly algebraically independent modulo $I(C)\cap \C[x_k, \ldots, x_n]$, and for all $j\geq k$ such that $j\in S$ there exists $t\in \{1, \ldots, m\}$ such that $j=i_t$ and by the construction in Step~\ref{step:avoid} we have $g_{S,t}\in I(C)\cap \C[x_k, \ldots, x_n]$.
    Analogously, for $\overline{S'}_k:=\overline{S'} \cap \{k, \ldots, n\}$, $\x_{\overline{S'}_k}$ is a transcendence basis
    of $\mathbf{C}[x_k, \ldots, x_n]$ modulo $I(C)\cap \C[x_k, \ldots, x_n]$.
    Hence
    \[
    \left\lvert \overline{S} \cap \{k, \ldots, n\} \right\rvert = \left\lvert \overline{S'} \cap \{k, \ldots, n\} \right\rvert
    \]
    for every $1 \leqslant k \leqslant n$. Thus, $S = S'$. Then varieties $\Rep(\Delta_c)$
    and $\Rep(\Delta_{c'})$ do not have common irreducible components due to 
    the specification of $\operatorname{TriangularZeroDim}$.
 \end{proof}
 

\section{Zero-dimensional triangular decomposition over rational functions}
\label{sec:lifting}

In this section we describe a slight modification of the zero dimensional equiprojectable triangular decomposition algorithm of \cite{Lifting} that was given over the field $K=\Q$, while here we work over the field $K=\C(\y)$ for $\y=y_1, \ldots, y_d$. 

\subsection{Equiprojectable decomposition}

\begin{definition}[\cite{Lifting}, p. 109] Given $h_1, \ldots, h_\ell \in K[z_1, \ldots, z_m]$ where $K$ is a field (here we use both $K=\C(\y)$ and $K=\C$),  assume that 
\[
  V = Z(h_1, \ldots, h_\ell) \subset \overline{K}^m
\]
is zero-dimensional, where $\overline{K}$ is the algebraic closure of $K$. 
Consider $\pi\colon\overline{K}^n\rightarrow \overline{K}^{n-1}$ the projection onto the first $n - 1$ coordinates, 
and for each $x \in V$ let $N(x):=\#\pi^{-1}(\pi(x))$, the number of points in $V$ in the $\pi$-fiber of $x$. 
Then decompose $V = V_1 \cup \cdots \cup V_d$ such that $V_i:=\{x\in V\;:\;N(x)=i\}$ for $i = 1, \ldots d$. 
Apply this splitting process recursively to each $V_1, \ldots, V_d$, using the fibers of the successive projections 
$\overline{K}^n\rightarrow \overline{K}^i$ onto the first $i$ coordinates,  for $i = n-2, \ldots, 1$. 
Thus we obtain a decomposition of $V$ into pairwise disjoint varieties that are each {\em equiprojectable}, 
which form the {\em equiprojectable decomposition} of $V$.
\end{definition}

The reason we consider the equiprojectable decomposition is because each equiprojectable component of $V$ is representable by a single triangular set with coefficients in $K$ (c.f. \cite[Section~2]{Lifting}). We use this fact to ensure that when lifting the equiprojectable components from  $\C$ to $\C[[\y]]$ (see below), the resulting triangular sets are reconstructable over $K=\C(\y)$. 

The main idea for computing an equiprojectable decomposition, encoded by triangular sets,  of  a  zero-dimensional affine variety  defined by polynomials  ${\bf H}=\{h_1, \ldots, h_\ell\}\subset \C(\y)[\z]$, is first to specify the variables $\y$ in a random point $\y^\ast\in \C^d$ such that the equiprojectable decomposition of $Z({\bf H})\subset \overline{\C(\y)}^m$, described by triangular sets in $\C(\y)[\z]$,  specializes to the equiprojectable decomposition of $Z({\bf H}_{\y=\y^\ast})\subset\C^m$. Then we can lift each triangular set in the equiprojectable decomposition of $Z({\bf H}_{\y=\y^\ast})\subset\C^m$ to a triangular set in $\C(\y)[\z]$  in the equiprojectable decomposition of $Z({\bf H})\subset \overline{\C(\y)}^m$.

To compute the equiprojectable decomposition, encoded by triangular sets,  of  a  zero-dimensional affine variety $Z({\bf H}_{\y=\y^\ast})\subset\C^m$, we cite the algorithm outlined in \cite[Section 4.]{Lifting}. Namely, they first call the zero-dimensional triangularization algorithm of \cite{Maza2000}, followed by the Split-and-Merge algorithm of \cite[Section 2.]{Lifting}.

\subsection{Lifting and reconstructing} \label{sec:lifting}

The lifting algorithm is a (slight extension) of the lifting procedure from~\citep[Section~4.2]{Schost2003b}.
We will work in the ring $\mathbb{C}[\mathbf{y}, \mathbf{z}]$ with $\mathbf{y} = (y_1, \ldots, y_d)$ 
and $\mathbf{z} = (z_1, \ldots, z_m)$ with $d + m = n$.
 
For a single lifting step we restate the specification of~\cite[Algorithm~\sc{Lift}]{Schost2003b} as follows:
\vspace{-.1cm}

{\small \begin{algorithm}[h]
\caption{$\operatorname{Lift}(\mathbf{H}, \mathbf{y}^\ast, s, \widetilde{\Delta})$}\label{alg:lift}  
  \begin{description}
  \item[Input] $\;$\\
 \vspace{-.1cm} \begin{enumerate}[leftmargin=\dimexpr\linewidth-7.5cm-\rightmargin\relax]
   \item set of polynomials ${\bf H} = \{ h_1(\mathbf{y}, \mathbf{z}), \ldots, h_m(\mathbf{y}, \mathbf{z})\} \subset \C[\mathbf{y}, \mathbf{z}]$;
   \item a point $\mathbf{y}^\ast \in \mathbb{C}^d$;
   \item a nonnegative integer $s$;
   \item a regular chain $\widetilde{\Delta} = \{ \widetilde{g}_1(\y,z_1),\widetilde{g_2}(\y,z_1,z_2), \ldots, \widetilde{g}_m(\y,\mathbf{z})\} \subset \mathbb{C}[\mathbf{y} - \mathbf{y}^\ast, \mathbf{z}]$ such that 
     \begin{itemize}[leftmargin=\dimexpr\linewidth-7.5cm-\rightmargin\relax]
      \item ${\rm lc}(\widetilde{g}_k)=1$ for $k=1, \ldots, m$
      \item $\widetilde{g}_k$ is reduced modulo $\{\widetilde{g}_1, \ldots, \widetilde{g}_{k-1}\}$
      \item there exists  $\Delta = \{ g_1(\mathbf{y}, z_1), g_2(\mathbf{y}, z_1, z_2), \ldots, g_m(\mathbf{y}, \mathbf{z}) \} \subset \mathbb{C}(\mathbf{y})[\mathbf{z}]$ such that 
      \begin{itemize}
        \item every element of $\mathbf{H}$ can be reduced to zero using $\Delta$
         \item $\widetilde{\Delta} \equiv \Delta \pmod{ (\mathbf{y} - \mathbf{y}^\ast)^{2^s} }$
      \end{itemize}
      \item $\operatorname{jac}_{\mathbf{z}}(\bf{H}|_{\mathbf{y} = \mathbf{y}^\ast})$ is invertible modulo $\widetilde{\Delta}|_{\mathbf{y} = \mathbf{y}^\ast}$;
     \end{itemize}
 \end{enumerate}
  \item[Output]  Regular chain $\widehat{\Delta} = \{ \widehat{g}_1({\bf y}, z_1), \ldots, \widehat{g}_m({\bf y}, {\bf z})\}\subset \C[\mathbf{y} - \mathbf{y}^\ast, {\bf z}]$ satisfying
      \begin{itemize}[leftmargin=\dimexpr\linewidth-7.5cm-\rightmargin\relax]
       \item ${\rm lc}(\widehat{g}_k) = 1$ for $k = 1, \ldots, m$;
       \item $\widehat{\Delta} \equiv \Delta \pmod{(\mathbf{y} - \mathbf{y}^\ast)^{2^{s + 1}}}$; 
     \end{itemize}
  \end{description}
  \end{algorithm}}
  Then one can adapt the general strategy of the main algorithm from~\cite[p.~113]{Lifting} to generalize the algorithm from~\cite[p.~584]{Schost2003} as shown in Algorithm~\ref{alg:hensel}.
         
         \begin{algorithm}[h]
         \caption{$\operatorname{TriangularZeroDim}$ via Hensel lifting}\label{alg:hensel}
         \begin{description}
           \item[Input] Set of polynomials $\mathbf{H} = \{h_1(\mathbf{y}, \mathbf{z}), \ldots, h_\ell(\mathbf{y}, \mathbf{z})\} \subset \mathbb{C}[\mathbf{y}, \mathbf{z}]$ such that $\operatorname{jac}_{\mathbf{z}}(\mathbf{H}_0)$ is invertible at every solution of $\mathbf{H}$ in $\overline{\mathbb{C}(\mathbf{y})}$, where $\mathbf{H}_0 = \{ h_1(\mathbf{y}, \mathbf{z}), \ldots, h_m(\mathbf{y}, \mathbf{z}) \}$.
           
           \item[Output] Equiprojectable triangular decomposition of $\mathbf{H}$ 
         \end{description}
         
         \begin{enumerate}[leftmargin=\dimexpr\linewidth-7.8cm-\rightmargin\relax]
           \item Pick coordinates of $\mathbf{y}_1, \mathbf{y}_2 \in \mathbb{C}^d$ randomly from a finite $\Gamma \subset \C$
           \item Compute the equiprojectable decompositions $\widetilde{\Delta}_1, \ldots, \widetilde{\Delta}_N$ and $\overline{\Delta}_1, \ldots, \overline{\Delta}_N$ of $\mathbf{H}|_{\mathbf{y} = \mathbf{y}_1}$ and $\mathbf{H}|_{\mathbf{y} = \mathbf{y}_2}$, respectively
           \item $s = 0$
           \item While not $\operatorname{Stop}$
           \begin{enumerate}
             \item $\widetilde{\Delta}_i := \operatorname{Lift}(\mathbf{H}_0, \mathbf{y}_1, s, \widetilde{\Delta}_i)$ for every $1 \leqslant i \leqslant N$
             \item s := s + 1
             \item $\Delta_i := \operatorname{RationalReconstruction}(\widetilde{\Delta}_i)$ for every $1 \leqslant i \leqslant N$
             \item if $\{\overline{\Delta}_1, \ldots, \overline{\Delta}_N\} = \{\Delta_1|_{\mathbf{y} = \mathbf{y}_2}, \ldots, \Delta_N|_{\mathbf{y} = \mathbf{y}_2}\}$, then $\operatorname{Stop} := \operatorname{true}$
           \end{enumerate}
           \item Return $\{\Delta_1, \ldots, \Delta_N\}$
         \end{enumerate}
         \end{algorithm}
  
  In Algorithm~\ref{alg:hensel}, we use the subroutine $\operatorname{RationalReconstruction}$ described in~\cite[Section~4.3.1]{Schost2003b}.

The following lemma is an adaptation of the arguments presented in \cite[Section 3]{Lifting} for the rational function field case. Since the bounds  in \cite[Section 3]{Lifting} were given in terms of  heights of rational numbers, we cannot straightforwardly cite those results, so we present the analogous bounds for $\C(\mathbf{y})$.

\begin{lemma} \label{lemma:degbound}
Assume that $\mathbf{H} = \{h_1(\mathbf{y}, \mathbf{z}), \ldots, h_\ell(\mathbf{y}, \mathbf{z})\} \subset \mathbb{C}[\mathbf{y}, \mathbf{z}]$ satisfies the input specifications of Algorithm \ref{alg:hensel}, and assume that $\deg_{(\z,\y)} (h_i)\leq {\mathcal H}$  for $i=1, \ldots, \ell$. Denote by $P$ the (finite) number of solutions of $\mathbf{H}$ in the algebraic closure $\overline{C(\y)}$ and by $Q$ the degree of the affine variety $Z(\mathbf{H})$ in $\C^{m+d}=\C^n$. Then there exists a polynomial $F\in \C[\y]$ with 
\[
\deg_\y(F)\leqslant 4m^2PQ{\mathcal H}^2
\]
with the property that if $F(y^*)\neq 0$ for $y^*\in \C^d$ then the equiprojectable triangular decomposition  of $\mathbf{H}$ specializes  to the equiprojectable triangular decomposition of $\mathbf{H}|_{\y=\y^*}$, and the Jacobian of $\mathbf{H}_0|_{\y=\y^*} $ does not vanish at any of the solutions of $\mathbf{H}|_{\y=\y^*}$ in $\C^m$. 
\end{lemma}

\begin{proof} Using the notation of \cite[Section 3]{Lifting}, for $k=1, \ldots, m$, denote by $u_k=u_{k, 1}z_1+ \cdots +u_{k, k}z_k$  ($u_{k,j}\in \C$) a primitive element for the projection of the finite number of solutions of $\mathbf{H}$ in the algebraic closure $\overline{\C(\y)}$, where  the projection is to the $(z_1, \ldots, z_k)$ coordinates. 
Let $\mu_k\in \C[\y][T]$ be the minimal polynomial of $u_k$.
Furthermore, let $ w_1, \ldots, w_m\in \C(\y)[T]$ be the parametrization of the solutions of $\mathbf{H}$ in  $\overline{\C(\y)}$ with respect to $u_m$. Analogously to  \cite[Lemmas 4-7]{Lifting}, we can  prove that the  3 hypotheses 
\begin{description}
\item[${H}_1$:]  None of the coefficients of $\mu_m, w_1, \ldots, w_m$ vanish at $\y=\y^*$; 
\item[${H}_2$:] { $\mu_k|_{\y=\y^*}\in \C[T]$ is squarefree for $k=1,\ldots,m$};
\item[${H}_3$:] {The  Jacobian of $\mathbf{H}_0|_{\y=\y^*} $  is not zero at the solutions of $\mathbf{H}|_{\y=\y^*}$};
\end{description}

\smallskip

\noindent
imply that the equiprojectable triangular decomposition  of $\mathbf{H}$ specializes  to the equiprojectable triangular decomposition of $\mathbf{H}|_{\y=\y^*}$, and the set of solutions of $\mathbf{H}$ in $\overline{\C(\y)}$ specializes at $\y=\y^*$ to the set of solutions of $\mathbf{H}|_{\y=\y^*}$ in $\C^m$. In \cite[Lemma 8]{Lifting} they prove that $H_1$ and $H_2$ is satisfied if for $a_k:={\rm Res}_T(\mu_k, \mu'_k)\in \C[\y]$ we have $a_k(\y^*)\neq 0$ for $k=1, \ldots, m$. Let $a=a_1\cdots a_m$.  Using that $\deg_T(\mu_k)\leq P$ and $\deg_\y(\mu_k)\leq  Q$ by \cite[Theorem 1]{Schost2003b}, we can see that  
 $$\deg_\y(a)\leq m(2P-1)Q.
 $$
Furthermore, let $J^h$ be the homogenization of the Jacobian of ${\bf H}_0$ by adding a homogenizing variable $z_0$ to $z_1, \ldots, z_m$, and consider $J^h(\mu'_m, v_1, \ldots, v_m)$ where $v_k:=w_k\mu'_m \mod \mu_m\in \C[\y][T]$ (note that this homogenization step turns the polynomials $w_i$ from having coefficients in $\C(\y)$ with higher degree numerators and denominators  into polynomials $v_i\in \C[\y][T]$ with $\y$-degrees at most $Q$, see \cite{Schost2003b} for more details). In \cite[Lemma 9]{Lifting} they prove that $H_3$ is satisfied if for the Sylvester resultant $b:={\rm Res}_T(J^h(\mu'_m, v_1, \ldots, v_m), \mu_m)\in \C[\y]$ we have $b(\y^*)\neq 0$. To bound the degree of $b$ first note that 
$$
\deg_T J^h(\mu'_m, v_1, \ldots, v_m)\leq mP{\mathcal H} 
$$and 
$$\deg_\y J^h(\mu'_m, v_1, \ldots, v_m)\leq mQ{\mathcal H}
$$
again by \cite[Theorem 1]{Schost2003b}, so 
$$
\deg_\y(b)\leq 2m^2QP{\mathcal H}^2.
$$
Putting it all together, for $F=ab$ we get the claimed degree bound.
\end{proof}

Now we are able to prove the last remaining part of our main theorem:

\begin{proof}[{\bf Proof of Theorem \ref{theorem:main}, \eqref{theorem(6)}.}]
Assume that $\deg_\x(f_i)\leq D$ for $i=1, \ldots, s$.   In Algorithm \ref{alg:main}, we have the following independent random uniform choices   from a    finite subset $\Gamma$ of  the coefficient field $\C$,
and we bound the probability of success using the Schwartz-Zippel lemma~\cite{Zippel:1979,Schwartz1980}.
\begin{itemize}[leftmargin=\dimexpr\linewidth-8cm-\rightmargin\relax]
\item In Step~\ref{step:equidim} we call the equidimensional decomposition algorithm of \cite{JeronimoSabia} with input $f_1, \ldots, f_s$. In \cite[Remark 10]{JeronimoSabia} they prove that the probability of success for their algorithm is at least
$$
1-\frac{c_1D^{n^2+n}+D^{c_2(n+1)}}{|\Gamma|}
$$
where $c_1$ and $c_2$ are constants.  They use randomization to obtain  linear combinations of
the input polynomials, changes of variables and the linear forms for the primitive elements used in each step.
\item In Step~\ref{step:square}, we choose at most $n + 1$  random linear  combinations 
$\widetilde{f}_1, \ldots, \widetilde{f}_{n+1}$ 
of the input polynomials $f_1, \ldots, f_s$. 
The correctness of the algorithm requires the assumption~\ref{A1} to hold. 
In \cite[Remark 4]{JeronimoSabia} they prove that this can be done with a probability of success
$$
\prod_{h=1}^{n+1}\left(1-\frac{D^{h-1}}{|\Gamma|}\right).
$$
\item In Step~\ref{step:avoid} we choose $\alpha\in \Gamma^n$ randomly, and we require Assumption~\ref{A2} to hold. 
Since $\deg_\x(\hat{g}_S) \leqslant (n+1)D^{n+1}$ (due to the degree bound for a Canny's resultant, see Section~\ref{subsec:Canny}), 
we have that the probability of success is at least
$$
1-\frac{2^n(n+1)D^{n+1}}{|\Gamma|}
$$
\item In Step~\ref{step:final} we assume that we use the randomized algorithm described in  Algorithm \ref{alg:hensel} with input ${\bf H}=\{h_1, \ldots, h_\ell\}$. We use Lemma \ref{lemma:degbound} to bound the probability of success. We have $\deg_\x(h_i)\leq {\mathcal H}=(n+1)D^{n+1}$, $P,Q\leq \deg{W}\leq D^n$, $m\leq n$, so we get that the probability of success is at least
$$
1-\frac{(n+1)^4D^{4(n+1)}}{|\Gamma|}.
$$
\end{itemize}
Since these random choices are independent, the probability of the success is the product of the individual probabilities, thus we get that the probability of the overall success of Algorithm \ref{alg:main} is at least the product of the above four probabilities, 
which, as long as $D\geq 2$, can be bounded from below by
$$
1-\frac{cD^{n^2+n}+ (n+1)^4D^{c'(n+1)}}{|\Gamma|},
$$
for some constants $c$ and $c'$, proving the claim. 
\end{proof}


\section{Examples}

To keep the presentation simple, in the following examples instead of random choices of numbers we use some choices which satisfy the requirements of the algorithm.

\begin{example}
This simple example demonstrates how our algorithm  avoids repetition of irreducible components when they need different sets of free variables. Let $n = 2$, $s = 1$, and $f_1 = x_1x_2(x_1 + x_2)$.
 Then
 \begin{enumerate}[label=(\arabic*),leftmargin=*]
   \item All irreducible components are of dimension one, so $\mathbf{p}_1 = \{f_1\}$ and $d_0 = d_1 = 1$. Note that for a hypersurface,  $\operatorname{EquiDim}$ from \cite{JeronimoSabia} always returns its defining equation.
   \item 
   \begin{enumerate}
     \item Since $s = 1$, $\widetilde{f}_1 = f_1$.
     \item $n - d_1 - 1 \leqslant |S| \leqslant n - d_0 - 1$ implies $S = \emptyset$.
     We set
     \[
     \hat{g}_{\emptyset} = \operatorname{PRes}_{\emptyset}(f_1) = f_1.
     \]
     \item We define $\widehat{\nabla}_{S_1} = \widehat{\nabla}_{S_2} = \{ \hat{g}_{\emptyset} \} = \{f_1\}$.
     \item We choose $\boldsymbol{\alpha} = (1, 1)$. Then
     \[
     \nabla_{S_1} := \{ f_1 \}, \;\; \nabla_{S_2} := \{ f_1|_{x_1 = 1} \} = \{ x_2(x_2 + 1) \}.
     \]
   \end{enumerate}
   \item The output is a union of
   \begin{itemize}[leftmargin=*]
   \item $\operatorname{TriangularZeroDim}(\{1\}, \{ f_1\}\cup\{ f_1\})$. Since $f_1$ alone is already a triangular set over $\C(\mathbf{x}_{\overline{\{1\}}}) = \C(x_2)$, we only have to make its coefficient coprime by division by $x_2$.
   Thus, the output is $\{x_1(x_1 + x_2)\}$.
   \item $\operatorname{TriangularZeroDim}(\{2\}, \{ x_2(x_2 + 1)\}\cup  \{f_1 \})$. This is equal to the gcd of $x_2(x_2 + 1)$ and $f_1$ as univariate polynomials in $x_2$, that is $x_2$.
   \end{itemize}
   So, the output is $\{x_1(x_1 + x_2)\}, \{x_2\}$.
 \end{enumerate}
\end{example}

\begin{example} Here, the algebraic set is the projective twisted cubic space curve (interpreted as the cone over it in $\C^4$). Since this curve is not a complete intersection, leading coefficients of its triangular set vanish on an extraneous projective curve, independently of the coordinate system. The intersection of this extraneous curve with the original twisted cubic will create embedded components that Triangularize in {\sc Maple} does not factor out. Here we show how our algorithm handles this example.   
Let $n = 4$, $s = 3$, and 
\[
(f_1, f_2, f_3) = (x_1x_3 - x_2^2,\; x_2^2 + x_2x_4 - x_3^2,\; x_1(x_2 + x_4) - x_2x_3).
\]
Then
\begin{enumerate}[label=(\arabic*),leftmargin=*]
  \item The system $f_1 = f_2 = f_3 = 0$ defines an irreducible two-dimensional variety, so $d_0 = d_1 = 2$. The output of  \cite{JeronimoSabia} is a set of at most $5$   polynomials defining this irreducible variety, so we can assume that $\mathbf{p}_2 = \{f_1, f_2, f_3\}$.
  \item \begin{enumerate}
    \item One can check that the choice $\widetilde{f}_1 = f_1$ and $\widetilde{f}_2 = f_2$ satisfies~\ref{A1}.
    \item Since $n - d_0 - 1 = n - d_1 - 1 = 3$, we compute four Canny's resultants, which turn out to be usual resultants in this case
    \begin{align*}
      \hat{g}_{\{1\}} &= x_2^2 + x_2x_4 - x_3^2,\\
      \hat{g}_{\{2\}} &= x_3(x_1x_4^2 - x_1^2 x_3 + 2x_1x_3^2 - x_3^3),\\
      \hat{g}_{\{3\}} &= x_2(x_1^2x_4 + x_1^2 x_2 - x_2^3),\\
      \hat{g}_{\{4\}} &= x_1x_3 - x_2^2.
    \end{align*}
    \item Using $\hat{g}_{\{1\}}, \ldots, \hat{g}_{\{4\}}$, we can define $\widehat{\nabla}_S$ for every two-element subset $S \subset \{1, 2, 3, 4\}$.
    In what follows, we will discuss only $S_1 = \{1, 2\}$ and $S_2 = \{1, 4\}$. 
    Other subsets will yield to the results similar to $S_2$.
    \item We can set $\boldsymbol{\alpha} = (1, 1, 1, 1)$.
    Then $g_{S_1, 1} = \hat{g}_{\{2\}}$, $g_{S_1, 2} = \hat{g}_{1}$, so $\nabla_{S_1} = \{ \hat{g}_{\{2\}}, \hat{g}_{\{1\}}\}$.
    For $S_2$ we have $g_{S_2, 1} = \hat{g}_{\{4\}}$ and
    \[
    g_{S_2, 2} = (x_2^2 + x_2x_4 - x_3^2)|_{x_2 = 1, x_3 = 1} = x_4.
    \]
  \end{enumerate}
  \item $\operatorname{TrianguarizeZeroDim}(S_1, \nabla_{S_1} \cup \{f_1, f_2, f_3\})$ will return a single triangular set 
  \[
  \{ x_3x_1^2 + x_2x_4 - x_3^2, x_2^2 + x_2x_4 - x_3^2 \}.
  \]
  One can see that $\operatorname{TrianguarizeZeroDim}(S_2, \nabla_{S_2} \cup \{f_1, f_2, f_3\})$ is empty, because
  \[
  f_2 - x_2g_{S_2, 2} = x_2^2 - x_3^2 \in \C[\mathbf{x}_{\overline{S_2}}].
  \]
  Analogously, all other calls of $\operatorname{TriangularizeZeroDim}$ will return empty sets.
\end{enumerate}

Although the ideal generated by $f_1, f_2, f_3$ in this example is prime, $\operatorname{Triangularize}$ function for {\sc RegularChains} library ({\sc Maple 2016}) returns two additional triangular sets, namely $\{x_1, x_2, x_3\}$ and $\{x_2, x_3, x_4\}$.

\end{example}

\begin{example}
Our last example demonstrates how Algorithm~\ref{alg:main} handles mixed dimensional algebraic sets with embedded components.  Let $n = 2$, $s = 2$, and
 \[
 (f_1, f_2) = \left( x_2(x_1 + x_2)(x_1^2 - 2), x_2(x_1 + x_2)(x_2^2 - 2)\right).
 \]
 Then
  \begin{enumerate}[label=(\arabic*),leftmargin=*]
   \item 
   The system $f_1 = f_2 = 0$ defines a union of two lines and two points, so $d_0 = 0$, $d_1 = 1$, and we assume that \cite{JeronimoSabia} returns
   \begin{align*}
     &\mathbf{p}_{0} = \{{(x_1-2x_2)^2-2, (x_1+x_2)^2-8, (x_1+2x_2)^2-18}\},\\
     &\mathbf{p}_{1} = \{x_2(x_1 + x_2)\}.
   \end{align*}
   \item 
   \begin{enumerate}
     \item The choice $\widetilde{f}_1 = f_1$ and  $\widetilde{f_2} = f_2$ satisfies~\ref{A1}.
     \item We compute
     \begin{align*}
       &\hat{g}_{\emptyset} = \operatorname{PRes}_{\emptyset}(f_1) = f_1,\\
       &\hat{g}_{\{1\}} = \operatorname{PRes}_{x_1}(f_1, f_2) = x_2^2 (x_2 - 1)(x_2^2 - 2)^3,\\
       &\hat{g}_{\{2\}} = \operatorname{PRes}_{x_2}(f_1, f_2) = 2(x_1 - 1)(x_1^2 + 2x_1 - 2)(x_1^2 - 2)^3,
     \end{align*}
     where $\hat{g}_{\{1\}}$ and $\hat{g}_{\{2\}}$ are actually perturbed resultants, because $f_1$ and $f_2$ are not coprime.
     \item We set $\widehat{\nabla}_{\{1, 2\}} = \{\hat{g}_{\{2\}}, \hat{g}_{\{1\}}\}$, $\widehat{\nabla}_{\{1\}} = \widehat{\nabla}_{\{2\}} = \{\hat{g}_{\emptyset}\}$.
     \item We chose $\boldsymbol{\alpha} = (1, 1)$.
     Then $\nabla_{\{1\}} = \widehat{\nabla}_{\{1\}}$,
     \[
     \nabla_{\{2\}} = \{ f_1|_{x_1 = 1} \} = \{-x_2(x_2 + 1)\},
     \]
     and $\nabla_{\{1, 2\}}$ is obtained from $\widehat{\nabla}_{\{1, 2\}}$ by taking squarefree parts
     \[
     \nabla_{\{1, 2\}} = \{ 2(x_1 - 1)(x_1^2 + 2x_1 - 2)(x_1^2 - 2), x_2(x_2 - 1)(x_2^2 - 2) \}.
     \]
   \end{enumerate}
   \item The output is a union of
   \begin{itemize}[leftmargin=*]
     \item $\operatorname{TriangularZeroDim}(\{1, 2\}, \nabla_{\{1, 2\}} \cup \mathbf{p}_{0})$.
     The output is $\{x_1 - x_2, x_2^2 - 2\}$.
     
     \item $\operatorname{TriangularZeroDim}(\{1\}, \nabla_{\{1\}} \cup \mathbf{p}_{1})$. 
     Since $\nabla_{\{1\}} \cup \mathbf{p}_{1}$ consists of $x_2(x_1 + x_2)(x_1^2 - 2)$ and $x_2(x_1 + x_2)$, the ideal is defined by a single polynomial $x_2(x_1 + x_2)$, which is already a triangular set itself.
     Dividing it by $x_2$, we make all its coefficients coprime as elements of $\C[x_2]$, so the output will be $\{x_1 + x_2\}$.
     
     \item $\operatorname{TriangularZeroDim}(\{2\}, \nabla_{\{2\}} \cup \mathbf{p}_{1})$.
     Since $\nabla_{\{2\}} \cup \mathbf{p}_{1}$ consists of $-x_2(x_2 + 1)$ and $x_2(x_1 + x_2)$, the result will consist of the gcd of these two polynomials as polynomials in $x_2$, that is $x_2$ itself.
     Hence, the output is $\{x_2\}$.
   \end{itemize}
   Thus, we obtain three triangular sets
   \[
   \{x_1 - x_2, x_2^2 - 2\}, \{x_1 + x_2\}, \{x_2\}.
   \]
  \end{enumerate}
  
  The output of $\operatorname{Triangularize}$ function from {\sc RegularChains} library ({\sc Maple 2016})
  consists of $\{x_1^2 - 2, x_2^2 - 2\}, \{x_1 + x_2\}, \{x_2\}$, so it contains embedded components.
\end{example}


\section{Concluding Remarks}

In a longer version of this paper we plan to further extend the results of this paper as follows:
\begin{itemize}[leftmargin=\dimexpr\linewidth-8.2cm-\rightmargin\relax]
\item Consider a modification of our algorithms that outputs squarefree regular chains that have degrees essentially bounded by $\deg W$. These triangular sets were studied for example in \cite{DahanSchost}, they are multiples of the ones our algorithm outputs, and have leading coefficients that depend on non-parametric variables. 
\item Modify the algorithm so that all intermediate degrees are also bounded by intrinsic geometric data of the input.
\item Consider algebraic sets defined by polynomials over $\Q$ and bound the height of the coefficients of the polynomials in the triangular sets. Such bit-size estimates were given for a single triangular set in the positive dimensional case in \cite{Dahanetal2012}.
\item Generalize Algorithm~\ref{alg:main} to the case when the input system contains inequations.
\end{itemize}

\section*{Acknowledgements}

Gleb Pogudin was supported by NSF grants CCF-0952591, CCF-1563942,  DMS-1413859, by PSC-CUNY grant
\#60098-00 48, by Queens College Research Enhancement, and by the
Austrian Science Fund FWF grant Y464-N18.
The authors would like to thank the Department of Mathematics at North Carolina State University and the Symbolic Computation Group for their hospitality during the visit of Gleb Pogudin in 2017,  allowing to make progress on this research.
The authors are grateful to the referees for the comments which helped to improve the paper.


\end{document}